\documentclass[10pt]{amsart}
\usepackage{amsfonts,amsmath,amscd,amssymb,amsthm,mathrsfs,esint}
\usepackage{color,fancyhdr,framed,graphicx,verbatim,mathtools}

\usepackage[margin=1in]{geometry}

\usepackage{tikz}
\usetikzlibrary{matrix,arrows,decorations.pathmorphing}

\newtheorem{theorem}{Theorem}
\theoremstyle{definition}

\newtheorem{proposition}[theorem]{Proposition}

\newtheorem{definition}[theorem]{Definition}

\newtheorem*{theorem*}{Theorem}
\newtheorem*{corollary*}{Corollary}
\newtheorem*{remark*}{Remark}
\newtheorem*{lemma*}{Lemma}

\title[Estimates for Green Operators on Compact Heisenberg
Manifolds]{Schatten and Sobolev Estimates for Green
  Operators on Compact Heisenberg Manifolds}
  
\author{Colin Fan} 
\address[Colin Fan]{Rutgers University--New Brunswick, Department of Mathematics, Piscataway, NJ 08854, USA}
\email{colin.fan@rutgers.edu}

\thanks{This work is supported by NSF DMS-1950102.}

\begin{document}
\begin{abstract}
  Let $M = \Gamma \setminus \mathbb{H}_d$
  be a compact quotient of
  the $d$-dimensional Heisenberg group $\mathbb{H}_d$ by a lattice
  subgroup $\Gamma$. We give Schatten and Sobolev
  estimates for the Green operator $\mathcal{G}_\alpha$
  associated to a fixed element of a  family of second order differential
  operators $\left\{ \mathcal{L}_\alpha \right\}$ on $M$. In particular, it follows that the Kohn Laplacian on functions on $M$ is subelliptic. Our main tool is Folland's description of
  the spectrum of $\mathcal{L}_\alpha$.
\end{abstract}

\maketitle

\section{Introduction}

\subsection{Motivation}
Associated to any CR manifold $M$ is the Kohn Laplacian, $\square_b$, which is densely defined on the space of square integrable $\left(p,q\right)$-forms.
The role of the Kohn Laplacian on a CR manifold $M$ is analogous to that of the Hodge Laplacian on a Riemannian manifold. Thus, taking inspiration from the fruitful methods of spectral analysis in the Riemannian setting, a natural consideration is whether the Kohn Laplacian encodes information about $M$. For example, there are CR analogs of Weyl's law and the Minakshisundaram asymptotic expansion that appear in Riemannian geometry \cite{Stanton}. A difficulty that arises in CR geometry is that the Kohn Laplacian is not elliptic -- a key feature of the Hodge Laplacian in the Riemannian setting.
However, when $M$ is strongly pseudoconvex and of hypersurface type, it is well known that the Kohn Laplacian is subelliptic on $q$-forms that are not functions or of top degree. For functions, the Kohn Laplacian need not be hypoelliptic, even with the assumption of strong pseudoconvexity (consider the Heisenberg group). In this paper, it will be shown that the Kohn Laplacian on functions is subelliptic in the case of compact Heisenberg manifolds
(compact quotients of the $d$-dimensional Heisenberg group by a lattice subgroup). Note that Sobolev gain estimates and subellipticity do not follow immediately from \cite{Kohn65} and \cite{Shaw2005HlderAl}, as estimates on functions are only obtained for non-pseudoconvex CR manifolds.

We consider compact Heisenberg manifolds because the
spectrum of the Kohn Laplacian for the Heisenberg group is not
discrete, but it is for compact Heisenberg manifolds. Moreover,
the spectrum on these manifolds was computed explicitly by
Folland in \cite{Folland2004CompactHM}, and therefore explicit calculations can be performed. In particular,
Folland computed the
spectrum of a family of second order differential operators $\left\{
  \mathcal{L}_\alpha 
\right\}$ that encodes
information about $\square_b$ on compact Heisenberg manifolds. Explicitly, if $\square_b$ acts
on $\left( 0,q\right)$-forms, then
\[\square_b \left( \sum_{\left| J \right| = q} f_J d
    \overline{z}^J\right) = \sum_{\left| J \right| = q}
  \mathcal{L}_{d - 2q} f_J d \overline{z}^J.\]

Using Folland's explicit spectral decomposition, we give
Schatten (Theorem \ref{thm:schatten}) and Sobolev norm (Theorem \ref{thm:sobolev}) estimates for the family of complex Green type operators $\left\{ \mathcal{G}_\alpha \right\}$
corresponding to $\left\{ \mathcal{L}_\alpha \right\}$. For brevity, we will refer to these operators as Green operators. In particular, when $\alpha = d$, we obtain estimates for the complex Green operator associated to the Kohn Laplacian on functions, implying the Kohn Laplacian is subelliptic and therefore hypoelliptic.

\subsection{Spectrum of  operators on compact
  Heisenberg manifolds}
The Heisenberg group comes equipped with two left-invariant
self-adjoint operators $\mathcal{L}_0$ and $i ^{-1} T$. These
operators strongly commute and therefore yield a joint
spectral decomposition for $L^2 \left( M \right) $. Moreover, the
family of operators given by $\mathcal{L}_\alpha = \mathcal{L}_0 +
i \alpha T$ for $\alpha\in \mathbb{R}$ can also play the role of
the standard Laplacian on the Heisenberg group as $\mathcal{L}_\alpha$ is also characterized by the symmetries of the Heisenberg group \cite{Stein}. We refer to \cite{Folland2004CompactHM} and \cite{fan2021tauberian} for further definitions
and details on compact Heisenberg manifolds, and the spectral
decomposition for $\mathcal{L}_\alpha$. We follow their notation, and we restate the main tools.

\begin{theorem*}[\cite{Folland2004CompactHM}]
  Let $M = \Gamma\setminus \mathbb{H}_d$ be a compact Heisenberg manifold and denote the center of $\Gamma$ by $\left(0,0,c\mathbb{Z} \right)$, $c>0$. Furthermore, let $\Lambda'$ be the dual lattice of the lattice $\Lambda = \pi \left(\Gamma\right)$, where $\pi: \mathbb{H}_d \to \mathbb{C}^d$ is the quotient map $\pi \left(z,t\right) = z$.
  The joint spectrum of $\mathcal{L}_0$ and $i ^{-1} T$ on $L^2
  \left( M\right)$ is
  \[\left\{ \left( \frac{ \pi \left| n \right|}{2c} \left( d + 2j
      \right), \frac{ \pi n }{2c} \right): j \in \mathbb{Z}_{\geq
      0}, n \in \mathbb{Z}\setminus \left\{ 0 \right\}\right\}
  \cup \left\{ \left( \frac{ \pi}{2} \left| \xi \right|^2,0
    \right): \xi \in \Lambda'\right\}\]
and the multiplicity of $ \left( \frac{ \pi \left| n \right|}{2c}
  \left( d + 2j \right), \frac{ \pi n }{2c} \right)$ is
\[\left| n\right|^d L \binom{j + d -
    1}{d - 1}\]
where $L$ is a constant determined by $\Gamma$. Moreover, the multiplicity of
an eigenvalue coming from the second set is dependent on the lattice
structure of $\Lambda'$.  
\end{theorem*}

Since $\mathcal{L}_0$ and $i^{-1} T$ are self-adjoint, and strongly commute, we obtain the following spectrum for $\mathcal{L}_\alpha$. 
\begin{corollary*}[\cite{Folland2004CompactHM}]
  For $\alpha\in \mathbb{R}$, the spectrum of $\mathcal{L}_\alpha$
  on $M$ is
  \[ \underbrace{\left\{ \frac{ \pi \left| n
          \right|}{2c} \left( d + 2j - \alpha\operatorname{sgn} n
  \right): j \in \mathbb{Z}_{\geq 0}, n \in \mathbb{Z}\setminus
  \left\{ 0 \right\}\right\}}_{\text{type } \left(a\right)} \cup
\underbrace{\left\{ \frac{ \pi}{2} \left|\xi \right|^2: \xi \in
    \Lambda'\right\}}_{\text{type } \left(b\right)}.\] 
\end{corollary*}
With this result, the family of Green operators $\left\{
  \mathcal{G}_\alpha 
\right\}$ corresponding to $\mathcal{L}_\alpha$ can be defined
explicitly.  Letting $\left\{ e_\ell
\right\}$ be an orthonormal basis for $\left( \ker
  \mathcal{L}_\alpha \right)^\perp$ induced from spectral decomposition of $\mathcal{L}_\alpha$, we have that each
$\mathcal{G}_\alpha: L^2 \left( M\right)\to L^2 \left( M\right)$
is densely defined by
\[\mathcal{G}_\alpha f = 0 \text{ if } f \in \ker \square_b\]
and
\[\mathcal{G}_\alpha f = \sum_\ell \frac{ \left\langle
      f,e_\ell
    \right\rangle}{\lambda_\ell^\alpha} e_\ell \text{ if } f\in
  \left( \ker \square_b\right) ^{\perp},\]
where $\lambda_\ell^\alpha$ is the eigenvalue corresponding to
$e_\ell$. We drop the $\alpha$ from the notation with the understanding that the eigenvalues that appear are dependent on $\alpha$.

To define a classical Sobolev space on a compact Heisenberg manifold $M$ we make use of a Laplace-Beltrami operator with an explicit spectrum that works well with the spectrum of $\mathcal{L}_\alpha$. One can also consider ``nonisotropic" Sobolev spaces, which we do not consider here. For such a study on the Heisenberg group, we refer to \cite{Stein} and \cite{CanarecciMasterthesis}.

\begin{theorem*}[\cite{Taylor1986}]
  There is a family of Riemannian metrics on $M$ parametrized by
  $\varepsilon > 0$ with associated (positive) Laplacian,
  \[L_\varepsilon = \mathcal{L}_0 - \varepsilon T^2.\]
\end{theorem*}

\begin{corollary*}
  For $\varepsilon>0$, the spectrum of $L_\varepsilon$ on $M$ is
  \[\left\{ \frac{ \pi \left| n \right|}{2c} \left( d + 2j
    \right) + \varepsilon\frac{ \pi^2}{4c^2}n^2: j\in
    \mathbb{Z}_{\geq 0}, n 
    \in \mathbb{Z}\setminus \left\{ 0 \right\}\right\}  \cup
  \left\{ \frac{ \pi}{2} \left| \xi \right|^2: \xi\in \Lambda'
  \right\}.\]
\end{corollary*}
We define $L = L_1$ and for the remainder of this paper consider $\alpha$ such that $-d \leq \alpha \leq d$.

\section{Schatten estimates for $\mathcal{G}_\alpha$}
In this section we characterize when the Green operator $\mathcal{G}_\alpha$ corresponding to $\mathcal{L}_\alpha$ has finite Schatten norm. In particular, this implies $\mathcal{G}_\alpha$ is compact for all $-d \leq \alpha \leq d$.
\begin{definition}
  Let $T$ be a compact and positive semi-definite operator from a
  separable Hilbert space to itself. For all $r\geq 1$, the {\it
    Schatten $r$-norm} of $T$ is defined by,
  \[\left\Vert T \right\Vert_r = \left( \sum_{k=0}^\infty
      \lambda_k \left( T \right)^r\right) ^{1/r}\]
  where $\lambda_k \left( T\right)$ are eigenvalues of $T$ ordered
  in decreasing fashion. 
\end{definition}

Note that we can define the Schatten norm for operators that are not necessarily compact, as long as the point spectrum is non-negative and countable. In particular, if there exists $r$ such that the Schatten $r$-norm is finite, then the operator must be compact. 

\begin{theorem}\label{thm:schatten}
  $\left\Vert \mathcal{G}_\alpha \right\Vert_r < \infty$ if and
  only if $r > d + 1$. 
\end{theorem}
\begin{proof}
  First assume $- d < \alpha < d$.
  We see that,
    \[\left\Vert \mathcal{G}_\alpha \right\Vert_r^r = \sum
      _{\substack{n\in \mathbb{Z}\setminus \left\{ 0 \right\} \\ j
          \in \mathbb{Z}_\geq 0}} \left| n \right|^d L \binom{j +d
        - 1}{d - 1} \left( \frac{ 2c}{\pi \left| n \right|\left( d
            +2j - \alpha \operatorname{sgn} \left(
              n\right)\right)}\right)^r + \sum_{\xi
        \in\Lambda'\setminus \left\{ 0
        \right\}}\frac{2^r}{\pi^r}\frac{ 1}{\left| \xi 
        \right|^{2r}}.\]
  We look at the first sum where $n < 0$. The case where $n > 0$
  follows similarly.
  We see that,
  \begin{align*}
    \sum_{n=1}^\infty \sum_{j=0}^\infty n^d \binom{j + d - 1}{d -
    1} \frac{ 1}{n^r \left( d + 2j + \alpha \right)^r}
    &= \sum_{n=1}^\infty \frac{ 1}{n^{r - d}} \sum_{j=0}^\infty
      \frac{ \left( j + d - 1 \right)!}{\left( d - 1 \right)! j!}
      \frac{ 1}{\left( d  + 2j + \alpha \right)^r}\\
    &= \frac{ 1}{\left( d - 1 \right)!} \sum_{n=1}^\infty \frac{
      1}{n^{r - d}}
      \sum_{j=0}^\infty \frac{ \left( j + d - 1 \right) \cdots
      \left( j + 1 \right)}{\left( d + 2j + \alpha \right)^r}.
  \end{align*}
  Note that the sum indexed by $j$ converges if and only if $r >
  d $. Similarly, the sum indexed by $n$ converges if and only
  if $r > d + 1$. Thus, the sum indexed by $n$ and $j$ converges
  if and only if $r > d + 1$. If we can show that the sum indexed
  by the lattice converges if and only if $r > d$ then our
  claim follows. This follows by noting that
  \[\sum_{\zeta\in \mathbb{Z}^{2d}\setminus \left\{ 0 \right\}}
    \frac{ 1}{\left| \zeta\right|^{2r}}\]
  converges if and only if $r > d$.

  The case where $\alpha = \pm d$ follows identically, as we only
  need to omit from the summation the case where $j = 0$ and
  either $n>0$ or $n < 0$. 
\end{proof}

\section{Sobolev estimates for $\mathcal{G}_\alpha$}
Let $\lambda_\ell$ be a non-zero eigenvalue of
$\mathcal{L}_\alpha$ and $\mu_\ell$ be the corresponding non-zero eigenvalue of $L$ that lives in the same eigenspace as
$\lambda_\ell$. For example,
\[\lambda_\ell = \frac{\pi \left|n\right|}{2c} \left(d + 2j - \alpha \operatorname{sgn} n \right)\text{ corresponds to } \mu_\ell = \frac{\pi \left|n\right|}{2c} \left(d + 2j\right) + \frac{\pi^2}{4c^2} n^2.\]
Note that this correspondence makes sense due to the simultaneous diagonalizability of $\mathcal{L}_0$ and $i^{-1} T$.
\begin{definition}
  Fix $s\in \mathbb{R}$. The {\it Sobolev space} $H^s \left(
    M\right)$ is defined to be
  \[H^s\left(M\right)=\left\{ f\in L^2 \left( M \right): \left( I + L \right)^{s/2}
      f \in L^2 \left( M  \right)\right\}.\]
  Moreover, we equip $H^s \left( M\right)$ with the norm
  \[\left\Vert f \right\Vert_s = \left\Vert \left( I + L
      \right)^{s/2} f \right\Vert_{L^2}.\]
\end{definition}

To show there exists $C_{\alpha,s}>0$ so that for any $t\in \mathbb{R}$, $\left\Vert \mathcal{G}_\alpha f \right\Vert_{t+s} \leq C \left\Vert f \right\Vert_t$, we only need to verify the values of $s$ for which the sequence $\left\{ \frac{ \left( 1 + \mu_\ell \right)^{s/2}}{\lambda_\ell}
\right\}$ bounded.

\begin{proposition}
  The sequence $\left\{ \frac{ \left( 1 + \mu_\ell
      \right)^{s/2}}{\lambda_\ell} \right\}$ is bounded if and
  only if $s\leq 1$. 
\end{proposition}
\begin{proof}
  If we restrict to type (b) eigenvalues we see that $\left\{
    \frac{ \left( 1 + \mu_\ell \right)^{s/2}}{\lambda_\ell}
  \right\}$ is bounded if and only if $s\leq 2$ as in this case,
  $\mu_\ell = \lambda_\ell$. If we can show
  that for type (a) eigenvalues the aforementioned sequence is
  bounded if and only if $s\leq 1$, then we are done.

  Assume $s > 1$. Consider the subsequence corresponding to $j =
  0$ and $n < 0$ if $\alpha \geq 0$, and $n > 0$ if $\alpha <
  0$. By convexity,
  \[\frac{ \left( 1 + \frac{ \pi d}{2c} \left| n \right| + \frac{
          \pi^2}{4c^2} n^2 \right){^s}}{\frac{ \pi^2}{4c^2}n^2
      \left( d + \left| \alpha \right| \right)^2} \geq \frac{ 1 +
      \left(\frac{\pi d}{2c}\right)^s \left| n \right|^s + \left(\frac{\pi^2}{4c^2}\right)^s n^{2s}}{\frac{\pi^2}{4c^2} \left(d+ \left|\alpha\right|\right)^2 n^2}.\]
  Since $s > 1$, the sequence is unbounded.

  Now assume $s \leq 1$. We can further assume $0 < s \leq 1$. By
  concavity,
  \[\frac{ \left( 1 + \frac{ \pi }{2c} \left( d + 2j \right)
        \left| n \right| + \frac{ \pi^2}{4c^2} n^2
      \right){^{s/2}}}{\frac{ \pi \left| n \right|}{2c}\left( d +
        2j - \alpha \operatorname{sgn} \left( n \right)\right)}
    \leq \frac{ 1 + \left(  \frac{ \pi }{2c}
        \left( d + 2j \right) \left| n
        \right|\right){^{s/2}} + \left| n \right|^s}{\frac{ \pi
        \left| n \right|}{2c}\left( d + 2j - \left| \alpha
        \right|\right)}.\]
  The right hand sequence is bounded.
\end{proof}

Sobolev norm estimates for $\mathcal{G}_\alpha$ follow
immediately. 
\begin{theorem}\label{thm:sobolev}
  There exists $C_\alpha$ independent of $f$ such that,
  \[ \left\Vert \mathcal{G}_\alpha f \right\Vert_{s+1} \leq C_\alpha
    \left\Vert f \right\Vert_s.\] 
\end{theorem}

In particular, the Kohn Laplacian on compact Heisenberg manifolds on functions is subelliptic, and therefore, hypoelliptic. 

One advantage of the spectral approach is obtaining sharp explicit constants $C_\alpha$ satisfying the above
inequality. 
\begin{proposition}
   For $-d < \alpha < d$, we can take
  \[C_\alpha = \sup_\ell \frac{ \left( 1 + \mu_\ell
      \right)^{1/2}}{\lambda_\ell} = \max \left\{ \frac{ \left(1 +
        \frac{ \pi}{2c}d + \frac{ \pi^2}{4c^2}\right)^{1/2}}{\frac{
          \pi}{2c}\left( d - \left| \alpha \right| \right)},
    \frac{ \left(1 + \frac{ \pi}{2}\left| \xi_0
        \right|^2\right)^{1/2}}{\frac{ 
        \pi}{2}\left| \xi_0 \right|^2}\right\},\]
  where $\xi_0$ is a non-zero vector in the lattice with minimal
  length. For $\alpha = \pm d$, we can take
  \[C_\alpha = \sup_\ell \frac{ \left( 1 + \mu_\ell
      \right)^{1/2}}{\lambda_\ell} = \max \left\{ \frac{ \left( 1
          + \frac{ \pi}{2c}\left( d + 2 \right) + \frac{
            \pi^2}{4c^2} \right)^{1/2}}{\frac{ \pi}{c}},\frac{
        \left(1 + \frac{ 
          \pi}{2}\left| \xi_0 \right|^2\right)^{1/2}}{\frac{ 
        \pi}{2}\left| \xi_0 \right|^2}\right\}.\]
\end{proposition}
\begin{proof}
  Fix $- d < \alpha < d$. We first show that $f:
  \mathbb{R}_{>0}\times \mathbb{R}_{\geq 
    0}\to \mathbb{R}$ defined by the rule,
  \[f \left( x,y\right) = \frac{ 1 + \frac{ \pi x}{2c} \left( d +
        2y \right) + \frac{ \pi^2}{4c^2}x^2}{\frac{ \pi^2
        x^2}{4c^2}\left( d + 2y - \left| \alpha \right|\right)^2}\]
  is decreasing in $x$, and decreasing in $y$. Fixing $y$, we see
  that 
  \[\frac{\partial }{\partial x} f \left( x,y \right) = - \frac{
      2c \left( \left( \pi d + 2 \pi y \right)x + 4c
      \right)}{\pi^2 \left( d + 2y  - \left| \alpha \right|
      \right)^2 x ^3}.\]
  Since $x > 0$, $f$ is decreasing in $x$. Now if we fix $x$, we
  see that
  \[\frac{\partial }{\partial y} f \left( x,y\right) = - \frac{ 4
      \left( 2\pi cxy + \pi^2 x^2 + \pi c\left(  d +  a 
        \right)x + 4c^2 \right)}{\pi^2 x^2 \left( d + 2y  - \left|
          \alpha \right| \right)^3}.\]
  Since $y \geq 0$, $f$ is decreasing in $y$. This yields the
  first case.

  For $\alpha = \pm d$, by a similar analysis to the above, it suffices to compare 
  \[\frac{ 1 + \frac{ \pi}{2c} d + \frac{ \pi^2}{4c^2}}{\frac{
        \pi^2}{4c^2} \left( d +  \left| \alpha \right| \right)^2}
    = \frac{ 1 + \frac{ \pi}{2c} d + \frac{ \pi^2}{4c^2}}{\frac{
        \pi^2}{c^2} d^2}
    \text{ and } \frac{ 1 + \frac{ \pi}{2c}\left( d + 2 \right) +
      \frac{ \pi^2}{4c^2}}{\frac{ \pi^2}{4c^2} \left( d + 2  -
        \left| \alpha \right|\right)^2} =  \frac{ 1 + \frac{
        \pi}{2c}\left( d + 2 \right) + 
      \frac{ \pi^2}{4c^2}}{\frac{ \pi^2}{c^2}}.\]
  Clearly the right hand side is larger than the left, giving the claim. 
\end{proof}

We note that one can also give Schatten and Sobolev norm estimates for the complex Green operator corresponding to the Kohn Laplacian on functions for lens spaces by a similar approach to \cite{kim2019sobolev}. In that setting, one must consider the existence of solutions to the diophantine system that appears in \cite{https://doi.org/10.48550/arxiv.2206.14250}.

\section*{Acknowledgements} 
First, we thank Yunus E. Zeytuncu for his support during this work.
We also thank Purvi Gupta and Jeffrey Im for their helpful
comments on an earlier version of this paper. This research was
completed at the REU Site: Mathematical Analysis and Applications
at the University of Michigan-Dearborn. We would like to thank the
National Science Foundation (DMS-1950102), the National Security
Agency (H98230-21), the College of Arts, Sciences, and Letters,
and the Department of Mathematics and Statistics for their
support.

\newcommand{\etalchar}[1]{$^{#1}$}

\end{document}